\documentclass[a4paper]{amsart}

\usepackage{amsmath,amstext,amssymb,mathrsfs,amscd,amsthm,indentfirst}
\usepackage{amsfonts}
\usepackage[small,nohug,heads=LaTeX]{diagrams}
\usepackage{pstricks}
\usepackage{verbatim}
\usepackage{enumerate}
\usepackage{graphicx}

\newtheorem{MainThm}{Theorem}
\newtheorem{MainCor}[MainThm]{Corollary}

\newtheorem{thm}{Theorem}[section]
\newtheorem{cor}[thm]{Corollary}
\newtheorem{lem}[thm]{Lemma}
\newtheorem{prop}[thm]{Proposition}
\theoremstyle{definition}
\newtheorem{defn}[thm]{Definition}
\theoremstyle{remark}
\newtheorem{rem}[thm]{Remark}

\numberwithin{equation}{section}

\newcommand{\bC}{\mathbb{C}}

\newcommand{\bP}{\mathbb{P}}
\newcommand{\bQ}{\mathbb{Q}}
\newcommand{\bR}{\mathbb{R}}

\newcommand{\bZ}{\mathbb{Z}}

\newcommand{\MT}[2]{\bold{MT #1}(#2)}
\newcommand{\MTtheta}{\bold{MT \theta}}

\newcommand\lra{\longrightarrow}
\newcommand\lla{\longleftarrow}

\newcommand\Diff{\mathrm{Diff}}

\newcommand\Bun{\mathrm{Bun}}

\newcommand{\hcoker}{/\!\!/}

\newarrow {Onto} ----{>>}
\newarrow {Equals} ===={=}

\newcommand{\CircNum}[1]{\ooalign{\hfil\raise .00ex\hbox{\scriptsize #1}\hfil\crcr\mathhexbox20D}}

\newcommand{\map}{\mathrm{map}}

\newcommand{\Fr}{\mathrm{Fr}}
\newcommand{\Spin}{\mathrm{Spin}}

\usepackage[hypertexnames=false]{hyperref}

\title{The space of immersed surfaces in a manifold}

\author{Oscar Randal-Williams}
\thanks{The author was supported by ERC Advanced Grant No.\ 228082, and the Danish National Research Foundation via the Centre for Symmetry and Deformation.}
\email{o.randal-williams@math.ku.dk}
\address{Institut for Matematiske Fag\\
Universitetsparken 5\\
DK-2100 K{\o}benhavn {\O}\\
Denmark}
\date{\today}


\begin{document}

\begin{abstract}
We study the cohomology of the space of immersed genus $g$ surfaces in a simply-connected manifold. We compute the rational cohomology of this space in a stable range which goes to infinity with $g$. In fact, in this stable range we are also able to obtain information about torsion in the cohomology of this space, as long as we localise away from $(g-1)$.
\end{abstract}
\maketitle

\section{Introduction and statement of results}
Let $M$ be a smooth manifold, not necessarily compact and possibly with boundary, and write $\mathring{M}$ for its interior. Let $\Sigma_g$ be a closed orientable surface of genus $g$. The \emph{space of immersed surfaces of genus $g$ in $M$} is defined to be the quotient space
$$\mathcal{I}_g(M) := \mathrm{Imm}(\Sigma_g, \mathring{M}) / \Diff^+(\Sigma_g),$$
so that a point in $\mathcal{I}_g(M)$ is represented by an unparametrised immersed oriented surface of genus $g$ in the interior of $M$. 

We propose to study the cohomology of this space. The differential topology of such spaces of unparametrised immersions has been studied in detail by Cervera--Mascar{\'o}--Michor \cite{CMM} and Michor--Mumford \cite{MichorMumford}, and the most elementary observation is that the action of $\Diff^+(\Sigma_g)$ on $\mathrm{Imm}(\Sigma_g, \mathring{M})$ is not free: for example, if an immersion is a covering space of its image, the group of covering transformations lies in the stabiliser. Thus the homotopy-type of $\mathcal{I}_g(M)$ is not as directly related to the homotopy-types of $\mathrm{Imm}(\Sigma_g, \mathring{M})$ and $\Diff^+(\Sigma_g)$ as one might like. To exert better homotopical control we may work with the homotopy quotient (i.e.\ Borel construction),
$$\mathcal{I}_g^h(M) := \mathrm{Imm}(\Sigma_g, \mathring{M}) \hcoker \Diff^+(\Sigma_g),$$
which enjoys better formal properties although it is perhaps not so geometrically meaningful. (In fact, $\mathrm{Imm}(\Sigma_g, \mathring{M})$ admits the structure on an infinite-dimensional smooth manifold, and the action of $\Diff^+(\Sigma_g)$ on $\mathrm{Imm}(\Sigma_g, \mathring{M})$ is smooth, proper, and has slices, so $\mathcal{I}_g(M)$ is the coarse space of the infinite-dimensional translation orbifold, and $\mathcal{I}_g^h(M)$ is the homotopy-type of this orbifold.) In order to relate this auxiliary space to the problem at hand, in Section \ref{app:1} we apply the theory developed by Cervera--Mascar{\'o}--Michor to prove the first part of the following theorem.

\begin{MainThm}\label{thm:Orbifold}
For $g \geq 2$, the map $H^*(\mathcal{I}_g(M);\bZ[\tfrac{1}{g-1}]) \to H^*(\mathcal{I}^h_g(M);\bZ[\tfrac{1}{g-1}])$ is an isomorphism. On the other hand, for every prime $p$ dividing $(g-1)$ the map $H^*(\mathcal{I}_g(M);\bZ_{(p)}) \to H^*(\mathcal{I}^h_g(M);\bZ_{(p)})$ is not surjective.
\end{MainThm}
\noindent Thus to study the cohomology of $\mathcal{I}_g(M)$ with $\bZ[\tfrac{1}{g-1}]$ coefficients it is enough to study the cohomology of $\mathcal{I}_g^h(M)$. We are mainly interested in rational cohomology, so this is no restriction, although our methods will also provide certain torsion information.

Our strategy is to compare the space $\mathcal{I}_g^h(M)$ with a certain universal space which is independent of $g$, much as in Madsen and Weiss' proof of the Mumford conjecture \cite{MW}. This space is an infinite loop space, and we may describe it as follows.

Let $\mathrm{Gr}_2^+(TM)$ denote the Grassmannian of oriented 2-planes in the tangent bundle of $M$. That is, a point is given by a point $x \in M$ and an oriented 2-plane $L \subset T_xM$. This space has an evident 2-plane bundle over it (with fibre over the point $(x, L)$ given by the vector space $L$) which we denote $\ell$, and hence a classifying map
$$\theta_M : \mathrm{Gr}_2^+(TM) \lra BSO(2).$$
Denote by $\MTtheta_M$ the Thom spectrum of the virtual bundle $-\ell \to \mathrm{Gr}_2^+(TM)$. There is a natural homomorphism
$$E : \pi_0(\Omega^\infty \MTtheta_M) \lra \bZ,$$
which we will describe in Section \ref{sec:components}, and we denote by $\Omega^\infty_{n} \MTtheta_M$ the collection of path components mapping to $n \in \bZ$.

\begin{MainThm}\label{thm:Main}
For any manifold $M$, there is a map
$$\alpha_M : \mathcal{I}^h_g(M) \lra \Omega^\infty_{1-g} \MTtheta_M.$$
If $M$ is simply-connected and of dimension at least three, the map induces an integral homology isomorphism in degrees 
$$* \leq \begin{cases}
\tfrac{2g-6}{5} & \text{if $\dim(M) = 3$}\\
\tfrac{2g-3}{3} & \text{if $\dim(M) > 3$}.
\end{cases}$$
\end{MainThm}

This theorem will be a consequence of the author's homology stability theorem for moduli spaces of surfaces with tangential structure \cite{R-WResolution}, along with the Hirsch--Smale theory of immersions \cite{Hirsch}. The Hirsch--Smale theory relates the space $\mathrm{Imm}(\Sigma_g, M)$ to a space of bundle-theoretic data, which we in turn identify with a space of $\theta_M$-reductions of $T\Sigma_g$. In order to apply the homology stability theorem of \cite{R-WResolution}, in Section \ref{sec:model} we give another model of $\mathcal{I}_g^h(M)$, which extends to the case of surfaces with boundary, and in Section \ref{sec:components} we compute the set of path components of this space.

Once homology stability is known, the methods of Galatius--Madsen--Tillmann--Weiss \cite{GMTW} identify the stable homology with that of the infinite loop space of the spectrum $\MTtheta_M$. Combining this theorem with Theorem \ref{thm:Orbifold}, we obtain the calculation

\begin{MainCor}\label{cor:Main}
There is an isomorphism of algebras
$$H^*(\mathcal{I}_g(M);\bZ[\tfrac{1}{g-1}]) \cong H^*(\Omega^\infty_{1-g} \MTtheta_M;\bZ[\tfrac{1}{g-1}])$$
in the stable range given by Theorem \ref{thm:Main}.
\end{MainCor}

For this theorem to be useful we must be able to compute the right-hand side. In general this is difficult, but for rational cohomology it is easy. Let us write $\Omega^\infty_\bullet\MTtheta_M$ for the basepoint component of the infinite loop space associated to the spectrum $\MTtheta_M$. There is a composition
$$H^{*+2}(\mathrm{Gr}_2^+(TM);\bQ) \cong H^*(\MTtheta_M;\bQ) \overset{\sigma}\lra H^*(\Omega^\infty_\bullet\MTtheta_M;\bQ)$$
where the first map is the Thom isomorphism and the second map is the cohomology suspension. The right-hand side is a graded commutative algebra and so this extends to a map
$$\Lambda(H^{*+2 > 2}(\mathrm{Gr}_2^+(TM);\bQ)) \lra H^*(\Omega^\infty_\bullet\MTtheta_M;\bQ)$$
from the free graded commutative algebra on the vector space of positive degree elements. This is an isomorphism.

We now give an application of this result, by describing $H^*(\mathcal{I}_g(\bR^n);\bQ)$ in the stable range. In order to do so we first describe a relationship between the cohomology of the moduli space of Riemann surfaces and the cohomology of $\mathcal{I}_g(M)$ for any manifold $M$.

\subsection{Relation to moduli spaces of curves}

Let $M$ be a manifold equipped with a Riemannian metric. There is then a continuous map
$$r : \mathcal{I}_g(M) \lra \mathcal{M}_g$$
from the space of unparamerised immersions to the coarse moduli space of Riemann surfaces of genus $g$. The map is defined as follows:

Let $i : \Sigma_g \looparrowright M$ be an immersion. The Riemannian metric on $M$ induces one on $\Sigma_g$, which in turn gives an almost complex structure on $\Sigma_g$. By the Newlander--Nirenberg theorem \cite{NewlanderNirenberg} and dimensional reasons (the Nijenhuis tensor has rank $(1,2)$), all almost-complex structures on a smooth surface are integrable, and so we obtain a unique complex structure on $\Sigma_g$, and hence a point in $\mathcal{M}_g$. This defines a function $\mathrm{Imm}(\Sigma_g, M) \to \mathcal{M}_g$, and it is easily checked to be continuous. Precomposing $i$ with a diffeomorphism of $\Sigma_g$ does not change this point, and so we get a well-defined map on the quotient space. 

The map $r$ depends on the original choice of Riemannian metric on $M$, but as the space of metrics is connected (in fact, contractible) it is well-defined up to homotopy, and endows $H^*(\mathcal{I}_g(M);\bQ)$ with the structure of an algebra over $H^*(\mathcal{M}_g;\bQ)$. Recall that there are defined the so-called Mumford--Morita--Miller classes \cite{Mumford}
$$\kappa_i \in H^{2i}(\mathcal{M}_g;\bQ)$$
and that by the theorem of Madsen and Weiss \cite{MW}, the map
$$\bQ[\kappa_1, \kappa_2, ...] \lra H^*(\mathcal{M}_g;\bQ)$$
is an isomorphism in degrees $* \leq \tfrac{2g-2}{3}$. We will typically describe $H^*(\mathcal{I}_g(M);\bQ)$ as an algebra over the polynomial ring $\bQ[\kappa_1, \kappa_2, ...]$.

\subsection{Immersions in $\bR^3$}\label{sec:3}

Let $i : \Sigma_g \looparrowright \bR^3$ be an immersion. Its derivative gives a bundle injection $Di : T\Sigma_g \hookrightarrow \epsilon^3$, and the orientation of $T\Sigma_g$ gives a trivialisation of the complementary bundle: we obtain a canonical isomorphism $T\Sigma_g \oplus \epsilon^1 \cong \epsilon^3$ and so a canonical Spin structure on $T\Sigma_g$. Recall \cite{Johnson} that Spin structures on oriented surfaces are classified by their $\bZ/2$-valued Arf invariant. This construction describes a map
$$\pi_0(\mathcal{I}_g(\bR^3)) \lra \bZ/2$$
sending an immersion to the Arf invariant of its associated Spin structure, and this is a bijection for $g\geq 1$. Let $\mathcal{I}_g(\bR^3)[\epsilon]$ denote the path component which maps to $\epsilon \in \bZ/2$. Then the unit map
$$\bQ \lra H^*(\mathcal{I}_g(\bR^3)[\epsilon];\bQ)$$
is an isomorphism in degrees $5* \leq 2g-6$.

\begin{rem}
In this case it is easy to extract a little torsion information as well: there is a surjection
$$\bZ/2 \oplus \bZ/24 \lra H_1(\mathcal{I}_g(\bR^3)[\epsilon];\bZ)$$
as long as $g \geq 6$, which is an isomorphism after inverting $(g-1)$. It would be interesting to know if this isomorphism is false before localising (for example, when $g=7$).
\end{rem}

\subsection{Immersions in $\bR^4$}

Let $i : \Sigma_g \looparrowright \bR^4$ be an immersion. Its derivative gives a bundle injection $Di : T\Sigma_g \hookrightarrow \epsilon^4$, with complement $V_i$. The orientation of $T\Sigma_g$ induces an orientation of $V_i$, and so there is defined an Euler class $e(V_i) \in H^2(\Sigma_g;\bZ)=\bZ$. This reduces modulo 2 to $w_2(V_i) = w_2(T\Sigma_g) = 0$, so is even. The assignment $i \mapsto \int_{\Sigma_g} e(V_i)$ gives a map
$$\pi_0(\mathcal{I}_g(\bR^4)) \lra 2\bZ$$
which we will show is a bijection for all $g \geq 0$, and we write $\mathcal{I}_g(\bR^4)[n]$ for the component which maps to $2n \in \bZ$. Then the map
$$\bQ[\kappa_1] \lra H^*(\mathcal{I}_g(\bR^4)[n];\bQ)$$
is an isomorphism in degrees $3* \leq 2g-3$ (and all higher $\kappa$ classes are zero).

\subsection{Immersions in $\bR^{2n+1}$, $2n+1 \geq 5$}\label{sec:2np1}

The space $\mathcal{I}_g(\bR^{2n+1})$ is connected for all $g \geq 0$, and the map
$$\bQ[\kappa_1, \kappa_2, ..., \kappa_{2n-2}] \lra H^*(\mathcal{I}_g(\bR^{2n+1});\bQ)$$
is an isomorphism in degrees $3* \leq 2g-3$ (and all higher $\kappa$ classes are zero). In Section \ref{sec:sseq} we use this calculation to study the behaviour of the spectral sequence
\begin{equation}\label{eq:sseqIntro}
E_2^{p,q} = H^p(\Gamma_g; H^q(\mathrm{Imm}(\Sigma_g, \bR^{2n+1});\bQ)) \Longrightarrow H^{p+q}(\mathcal{I}^h_g(\bR^{2n+1});\bQ)
\end{equation}
in the stable range, where we find a curious pattern of differentials.

\subsection{Immersions in $\bR^{2n}$, $2n \geq 6$}\label{sec:2n}

The space $\mathcal{I}_g(\bR^{2n})$ is connected for all $g \geq 0$. Let $\pi : E \to \mathcal{I}_g(\bR^{2n})$ be the universal family of surfaces. It is tautologically equipped with a map $i : E \to \bR^{2n}$ which is an immersion on each fibre. The derivative gives a bundle injection $Di : T_\pi E \hookrightarrow \epsilon^{2n}$ of the vertical tangent bundle, and we write $V \to E$ for its $(2n-2)$-dimensional complement, which has an orientation induced by the orientation of $T_\pi E$. Define a cohomology class
$$\Delta := \pi_!(e(V)) \in H^{2n-4}(\mathcal{I}_g(\bR^{2n});\bQ).$$
Then the map
$$\bQ[\kappa_1, \kappa_2, ..., \kappa_{2n-3}, \Delta] \lra H^*(\mathcal{I}_g(\bR^{2n});\bQ)$$
is an isomorphism in degrees $3* \leq 2g-3$ (and all higher $\kappa$ classes are zero).

\subsection{Outline}

In Section \ref{app:1} we describe the relationship between $\mathcal{I}_g(M)$ and the homotopical version $\mathcal{I}_g^h(M)$, and show that they have isomorphic homology after inverting $(g-1)$. We also show that these spaces have different cohomology when localised at primes dividing $(g-1)$. In Section \ref{sec:model} we relate the space $\mathcal{I}_g^h(M)$ to a moduli space of surfaces with tangential structure, as in \cite{R-WResolution}, which we define for surfaces with boundary. In Section \ref{sec:components}, which is the bulk of the paper, we calculate the set of path components of spaces of immersions of surfaces (possibly with boundary) into a manifold $M$. This allows us to verify the conditions of the homology stability theorem of \cite{R-WResolution}, and we then give the stability range and describe the stable homology. In Section \ref{sec:CalcEuc} we first give the details of the calculations of $\mathcal{I}_g(\bR^n)$ described above, and then in Section \ref{sec:sseq} we give a calculation of the spectral sequence (\ref{eq:sseqIntro}) in the stable range.

\section{The orbifold structure of $\mathcal{I}_g(M)$ and proof of Theorem \ref{thm:Orbifold}}\label{app:1}

The work of Cervera--Mascar{\'o}--Michor \cite{CMM} establishes that the group $\Diff^+(\Sigma_g)$ acts on $\mathrm{Imm}(\Sigma_g, M)$ properly, so in particular with finite stabiliser groups. In fact, they show that stabiliser group $\Diff^+(\Sigma_g)_i$ of an immersion $i$ acts freely and properly discontinuously on $\Sigma_g$, so it is finite and
$$2-2g = \chi(\Sigma_g) =  \# \{\Diff^+(\Sigma_g)_i\} \cdot \chi(\Sigma_g / \Diff^+(\Sigma_g)_i).$$
As $\Sigma_g / \Diff^+(\Sigma_g)_i$ is an orientable surface it has even Euler characteristic, and so for $g \geq 2$ we deduce that
$$\# \{\Diff^+(\Sigma_g)_i\} \mid (g-1).$$
Thus the fibres of $\pi: \mathcal{I}^h_g(M) \to \mathcal{I}_g(M)$ are all classifying spaces of finite groups of order dividing $(g-1)$, and hence are $\bZ[\tfrac{1}{g-1}]$-acyclic. In order for this to imply that $\pi$ is a $\bZ[\tfrac{1}{g-1}]$-cohomology isomorphism we also require that the map $\pi$ be locally well-behaved in a suitable sense.

Such a sense is provided by the construction in \cite{CMM} of slices for the action of $\Diff^+(\Sigma_g)$ on $\mathrm{Imm}(\Sigma_g, M)$. For each $i \in \mathrm{Imm}(\Sigma_g, M)$ let $\mathcal{N}(i)$ denote the normal bundle of the immersion (formed using a metric on $M$ which we fix once and for all). They construct a submanifold $\mathcal{Q}(i) \subset \mathrm{Imm}(\Sigma_g, M)$ diffeomorphic to a convex open neighbourhood of zero in the space $\Gamma(\mathcal{N}(i))$ of smooth sections of the normal bundle of the immersion $i$, enjoying the following properties (cf.\ \cite[\S 2.4]{MichorMumford}):
\begin{enumerate}[(i)]
	\item $\mathcal{Q}(i)$ is invariant under the isotropy group of $i$,
	\item \label{it:ii}$\varphi(\mathcal{Q}(i)) \cap \mathcal{Q}(i) \neq \emptyset$ if and only if $\varphi$ is in the  isotropy group of $i$,
	\item $\Diff^+(\Sigma_g) \cdot \mathcal{Q}(i)$ is an open, invariant neighbourhood of the orbit $\Diff^+(\Sigma_g) \cdot \{i\}$, and retracts onto it.
\end{enumerate}

Using that $\mathcal{Q}(i)$ is homeomorphic to a convex open subset of the vector space $\Gamma(\mathcal{N}(i))$, we see that scaling vectors in $\Gamma(\mathcal{N}(i))$ gives a $\Diff^+(\Sigma_g)_i$-equivariant deformation retraction of $\mathcal{Q}(i)$ onto $\{i\}$, and so by (\ref{it:ii}) a $\Diff^+(\Sigma_g)$-equivariant deformation retraction of $\Diff^+(\Sigma_g) \cdot \mathcal{Q}(i)$ onto $\Diff^+(\Sigma_g) \cdot\{i\}$. Let $U :=[\Diff^+(\Sigma_g) \cdot \mathcal{Q}(i)] \subset \mathcal{I}_g(M)$, a contractible open neighbourhood of $[i]$.


For $n > 0$, consider the presheaf $R^n\pi_*\bZ[\tfrac{1}{g-1}]$ on $\mathcal{I}_g(M)$ given by
$$V \mapsto H^n(\pi^{-1}(V);\bZ[\tfrac{1}{g-1}]).$$
The stalk at $[i]$ of this presheaf is computed as a colimit over neighbourhoods of $[i]$, but as $\pi^{-1}(U)$ deformation retracts onto $\pi^{-1}([i])$, the stalk is the same as the $n$-th cohomology of $\pi^{-1}([i])$. This fibre is $E\Diff^+(\Sigma_g)\times_{\Diff^+(\Sigma_g)} \Diff^+(\Sigma_g)_i \simeq B\Diff^+(\Sigma_g)_i$ the classifying space of a finite group of order dividing $(g-1)$, and hence has trivial $\bZ[\tfrac{1}{g-1}]$-cohomology in positive degrees. Hence all stalks of the presheaf $R^n\pi_*\bZ[\tfrac{1}{g-1}]$ are trivial, and hence the sheafification is trivial. Thus the Leray spectral sequence for the map $\pi$ in $\bZ[\tfrac{1}{g-1}]$-cohomology collapses, which establishes the first part of Theorem \ref{thm:Orbifold}. We remark that this is entirely analogous to the proof of the same result for finite-dimensional orbifolds.

\vspace{2ex}

We will now prove the second part of Theorem A. Let $G =
\bZ/(g-1)$, which has a free action on $\Sigma_g$ with quotient diffeomorphic to $\Sigma_2$ (by viewing $\Sigma_g$ as a torus with $(g-1)$ smaller tori glued on at regular intervals around a meridian, where $G$ acts by rotation).
We then have a map
$$E = EG \times_G \Sigma_g \lra \{*\} \times_G \Sigma_g \cong \Sigma_2$$
which is a submersion on each fibre, so choosing an immersion $\Sigma_2 \looparrowright M$ gives a fibrewise immersion of the surface bundle $E \to BG$ into $M$, which is classified by a map $f : BG \to \mathcal{I}_g^h(M)$. Its composition with $\mathcal{I}_g^h(M) \to
\mathcal{I}_g(M)$ is constant, taking value the equivalence class of the immersion $\Sigma_g \to \Sigma_g/G \cong \Sigma_2 \looparrowright M$.

Let $\pi \subset G$ be a subgroup of order $p$. We claim that the map $B\pi \to BG \to \mathcal{I}_g^h(M)$ is non-trivial on integral cohomology in positive degrees: as the cohomology of $\pi$ is $p$-local in positive degrees, it follows that there
are classes in $H^*(\mathcal{I}_g^h(M);\bZ_{(p)})$ which do not come from $H^*(\mathcal{I}_g(M);\bZ_{(p)})$. To establish this claim we consider the
composition
$$B\pi \lra BG \overset{f}\lra \mathcal{I}_g^h(M) \lra *\hcoker \Diff^+(\Sigma_g) = B\Diff^+(\Sigma_g) \lra B\Gamma_g$$
where $\Gamma_g = \pi_0(\Diff^+(\Sigma_g))$ is the \emph{mapping class
group} of $\Sigma_g$. On fundamental groups the homomorphism $G \to
\Gamma_g$ is injective (which may be seen, for example, by considering
the action of $G$ on the first homology of $\Sigma_g$). It is well known that
$\Gamma_g$ has a torsion-free normal subgroup of finite index, say
$\Delta$, so $\pi \to G \to \Gamma_g \to \Gamma_g / \Delta$ is also injective.
It follows from a theorem of Swan \cite[Theorem 1]{Swan} that this composition is non-trivial
on integral cohomology in infinitely many positive degrees, which proves the claim.

\section{Other homotopical models and surfaces with boundary}\label{sec:model}

Let us write $\Bun(T\Sigma_g, \ell)$ for the set of bundle maps $T\Sigma_g \to \ell$, i.e.\ those continuous maps which are linear isomorphisms on each fibre, and equip it with the compact-open topology. We will construct a slightly different homotopical model to $\mathcal{I}_g^h(M)$. Define
$$\mathcal{M}^{\theta_M}(\Sigma_g) := \mathrm{Bun}(T\Sigma_g, \ell) \hcoker \Diff^+(\Sigma_g)$$
to be the homotopy quotient, or Borel construction.

\begin{lem}
There is a (na\"{i}ve) $\Diff^+(\Sigma_g)$-equivariant weak homotopy equivalence
$$\mathrm{Imm}(\Sigma_g, \mathring{M}) \simeq \mathrm{Bun}(T\Sigma_g, \ell).$$
\end{lem}
Here, by a \emph{na\"{i}ve equivariant weak homotopy equivalence} we mean that the two $\Diff^+(\Sigma_g)$-spaces are connected by a zig-zag of equivariant maps which are (non-equivariant) weak homotopy equivalences. This relation is too coarse for many applications of equivariant homotopy theory, but is sufficient to guarantee that the homotopy quotients are weakly equivalent.

\begin{proof}
We proceed in two steps. The first step is to note that Hirsch--Smale immersion theory implies that the $\Diff^+(\Sigma_g)$-equivariant ``derivative map"
$$\mathrm{Imm}(\Sigma_g, \mathring{M}) \lra \mathrm{Bun}_{inj}(T\Sigma_g, TM)$$
to the space of injective bundle maps is a weak homotopy equivalence.

The second step is to note that there is a $\Diff^+(\Sigma_g)$-equivariant map 
$$\mathrm{Bun}_{inj}(T\Sigma_g, TM) \lra \mathrm{Bun}(T\Sigma_g, \ell)$$
sending the bundle injection $e : T\Sigma_g \hookrightarrow TM$ to the bundle map
\begin{eqnarray*}
e' : T\Sigma_g &\lra& \ell\\
(x, v) & \mapsto & ((e(x), e(T_x\Sigma_g) \subset T_{e(x)}M), e(v)).
\end{eqnarray*}
This is easily seen to be a homeomorphism, and the claim follows.
\end{proof}

By taking homotopy quotients, this lemma shows that there is a weak homotopy equivalence $\mathcal{I}^h_g(M) \simeq \mathcal{M}^{\theta_M}(\Sigma_g)$. Along with Theorem \ref{thm:Orbifold} this implies the zig-zag
$$\mathcal{I}_g(M) \overset{\simeq_{H_*}}\longleftarrow \mathcal{I}_g^h(M) \overset{\simeq}\lra \mathcal{M}^{\theta_M}(\Sigma_g)$$
where the leftwards map is a $\bZ[\tfrac{1}{g-1}]$-homology equivalence and the rightwards map is a weak homotopy equivalence.

\subsection{Surfaces with boundary}

The definition of the spaces $\mathcal{M}^{\theta_M}(\Sigma_g)$ extends easily to the case of surfaces with boundary. Let us write $\Sigma_{g,b}$ for a surface of genus $g$ with $b$ boundary components. We fix a bundle map $\delta : T\Sigma_{g,b}\vert_{\partial \Sigma_{g, b}} \to \ell$ and let $\Bun_\partial(T\Sigma_{g, b}, \ell;\delta)$ be the space of those bundle maps which restrict to $\delta$ on the boundary. The group $\Diff_\partial^+(\Sigma_{g,b})$ of diffeomorphisms which are the identity near $\partial \Sigma_{g,b}$ acts on this space, and we define
$$\mathcal{M}^{\theta_M}(\Sigma_{g, b};\delta) := \Bun_\partial(T\Sigma_{g, b}, \ell;\delta) \hcoker \Diff^+_\partial(\Sigma_{g,b}).$$

The main results of \cite{R-WResolution} reduce the problem of establishing homology stability for the moduli spaces $\mathcal{M}^\theta(\Sigma_{g, b};\delta)$ to the problem of understanding the sets $\pi_0(\mathcal{M}^\theta(\Sigma_{g,b};\delta))$ and the gluing maps between them sufficiently well. 

\subsection{Spaces of bundle maps as spaces of lifts}
In order to understand these sets of path components, we require a further model for the space $\Bun_\partial(T\Sigma_{g, b}, \ell;\delta)$.  Let us pick a map $\tau : \Sigma_{g,b} \to BSO(2)$ classifying the tangent bundle (i.e. we have a given bundle isomorphism $\varphi: \tau^* \gamma_2^+ \cong T\Sigma_{g,b}$). 

\begin{defn}
For a fibration $\theta : X \to BSO(2)$ let $\mathrm{Lifts}(\tau, \theta)$ denote the space of maps $l : \Sigma_{g, b} \to X$ such that $\tau = \theta \circ l$. 
If a lift $b : \partial \Sigma_{g,b} \to X$ of $\tau\vert_{\partial \Sigma_{g,b}}$ is already given, let $\mathrm{Lifts}(\tau, \theta;b)$ denote the subspace of those $l$ which restrict to $b$ on the boundary.
\end{defn}

If $\theta$ is any map, we define $\mathrm{Lifts}(\tau, \theta;b)$ to be $\mathrm{Lifts}(\tau, \theta^f;b^f)$ where $\theta^f : X^f \to BSO(2)$ is the canonical replacement of $\theta$ by a fibration, and $b^f$ is the composition of $b$ with the canonical map $X \to X^f$. If $\theta$ is already a fibration then the spaces $\mathrm{Lifts}(\tau, \theta;\delta)$ and $\mathrm{Lifts}(\tau, \theta^f;b^f)$ are homotopy equivalent.

Suppose $\theta$ is a fibration (by replacing it if necessary). Given a lift $l$ of $\theta$, we have a bundle map
$$\mathcal{B}(l) : T\Sigma_{g, b} \overset{\varphi}\cong \tau^*\gamma_2^+ = (\theta \circ l)^*\gamma_2^+ \cong l^*(\theta^* \gamma_2^+) \lra \theta^*\gamma_2^+,$$
and this construction defines a map
$$\mathrm{Lifts}(\tau, \theta;b) \lra \Bun_\partial(T\Sigma_{g, b}, \theta^*\gamma_2^+;\mathcal{B}(b)),$$
which is a weak homotopy equivalence.

\subsection{Gluing}

Suppose we are given a surface $\Sigma$, a collar $c : [0,1) \times \partial \Sigma \to \Sigma$ and a boundary condition $\delta : T\Sigma\vert_{\partial \Sigma} \to \ell$, and similar data $(\Sigma', c', \delta')$. Suppose we have embeddings
$$\partial \Sigma \overset{i}\lla \partial_0 \overset{i'}\lra \partial \Sigma'$$
such that
$$\epsilon^1 \oplus T\partial_0 \overset{\epsilon^1 \oplus Di}\lra \epsilon^1 \oplus T\partial\Sigma = T([0,1) \times \partial \Sigma)\vert_{\{0\} \times \partial \Sigma} \overset{Dc}\lra T\Sigma\vert_{\partial \Sigma} \overset{\delta}\lra \ell$$
and the analogous map for $(\Sigma', c', \delta', i')$ are equal. Then there is a gluing map
$$\Bun_\partial(T\Sigma, \ell;\delta) \times \Bun_\partial(T\Sigma', \ell;\delta') \lra \Bun_\partial(T(\Sigma \cup_{\partial_0} \Sigma');\delta \cup \delta')$$
and an associated gluing map
$$\mathcal{M}^{\theta_M}(\Sigma;\delta) \times \mathcal{M}^{\theta_M}(\Sigma';\delta') \lra \mathcal{M}^{\theta_M}(\Sigma \cup_{\partial_0} \Sigma';\delta \cup \delta').$$

%

\section{Isotopy classes of immersions and homology stability}\label{sec:components}

Fix an immersion $\delta : [0,1) \times \partial \Sigma_{g,b} \looparrowright M$ and let us write $\delta$ for the associated bundle map $T\Sigma_{g,b}\vert_{\partial \Sigma_{g,b}} \to \ell$ as well as for the underlying map $\Sigma_{g, b} \to \mathrm{Gr}_2^+(TM)$; the precise meaning will be clear from the context. We first aim to compute
$$\pi_0(\Bun_\partial(T\Sigma_{g, b}, \ell;\delta)) = \pi_0(\mathrm{Lifts}(\tau, \theta_M;\delta))$$
as then $\pi_0(\mathcal{M}^{\theta_M}(\Sigma_{g, b};\delta))$ will be the quotient of this set by the evident action of the \emph{mapping class group}
$$\pi_0(\Diff_\partial^+(\Sigma_{g,b})) =: \Gamma_{g,b}.$$

From now on we make the assumption that $M$ is simply-connected and of dimension at least three, which implies that all boundary conditions $\delta$ are regularly homotopic, and we can hence take them to be in some standard position. It also implies that the map
$$\pi_0(\map_\partial(\Sigma_{g,b}, M;{\delta})) \lra H_2(M;\bZ),$$
which closes up the boundaries (in a standard position) with standard discs and takes the fundamental class, is a bijection.

\begin{prop}\label{prop:5mfld}
Let $M$ be simply-connected and of dimension at least 5. Then the natural map 
$$\pi_0(\Bun_\partial(T\Sigma_{g,b}, \ell;\delta)) \lra \pi_0(\map_\partial(\Sigma_{g,b}, M;{\pi_M \circ \delta})) \cong H_2(M;\bZ)$$
induces a bijection. The action of the mapping class group is trivial.
\end{prop}
\begin{proof}
Consider the diagram
\begin{diagram}[LaTeXeqno]\label{dig:1}
S^{d-2}\\
\dTo\\
\mathrm{Fr}_2(\bR^d) & \rTo & \mathrm{Gr}_2^+(TM) & \rTo^{\theta_M \times \pi_M}& BSO(2) \times M\\
\dTo^{\mathrm{ev}_1} & & \dTo^{\theta_M} & & \dTo^{\pi_1}\\
S^{d-1} & & BSO(2) & \rEq & BSO(2)
\end{diagram}
where the first column and middle row are homotopy fibre sequences. By the first column, if $d \geq 5$ then $\mathrm{Fr}_2(\bR^d)$ is 2-connected, and hence the map $\theta_M \times \pi_M$ is 3-connected. Fixing a map $\tau : (\Sigma_{g,b}, \partial \Sigma_{g,b}) \to (BSO(2),*)$ classifying the tangent bundle, the map
$$\mathrm{Lifts}(\tau, \theta_M;\delta) \lra \mathrm{Lifts}(\tau, \pi_1; (\theta_M \times \pi_M)\circ \delta) \simeq \map_\partial(\Sigma_{g,b}, M;\pi_M \circ \delta)$$
is then 1-connected and in particular a bijection on $\pi_0$.
\end{proof}

The cases of 3- and 4-dimensional background manifolds are rather more complicated. In these cases an immersion of a surface into such a manifold endows the surface with additional geometric structure which cannot be recovered from the homotopy class of the map alone. In dimension 3 this is a Spin structure, and in dimension 4 it is a choice of oriented rank 2 vector bundle with Euler class satisfying a certain congruence condition.

\begin{prop}\label{prop:dim3}
Let $M$ be a simply-connected 3-manifold; such a manifold admits a Spin structure. A choice of Spin structure on $M$ gives a bijection
$$\pi_0(\Bun_\partial(T\Sigma_{g,b}, \ell;\delta)) \cong \mathrm{Spin}(\Sigma_{g,b};*) \times H_2(M;\bZ),$$
where $\mathrm{Spin}(\Sigma_{g,b};*)$ denotes the set of isomorphism classes of Spin structures on $\Sigma_{g,b}$, with the trivial Spin structure around the boundary. The action of the mapping class group is given by its usual action on the set of Spin structures.

\end{prop}
\begin{proof}
An orientable 3-manifold admits a Spin structure as there is a relation $w_2 = w_1^2$ among its Stiefel--Whitney classes by Wu's formula. We choose one, $\mathfrak{s}$, once and for all. 

From diagram (\ref{dig:1}) in the proof of Proposition \ref{prop:5mfld}, we see there are homotopy cartesian squares
\begin{diagram}
& & \mathrm{Lifts}(\tau, \theta_M; \delta) & \lTo & \mathrm{Lifts}(\tau \times f, \theta_M \times \pi_M; \delta)\\
& & \dTo & & \dTo\\
\map_\partial(\Sigma_{g,b}, M;\pi_M \circ \delta) & \,\,\simeq\,\, &  \mathrm{Lifts}(\tau, \pi_1; \pi_M \circ \delta) & \lTo & \{f : \Sigma_{g, b} \to M\}
\end{diagram}
for each point $\{f\} \in \mathrm{Lifts}(\tau, \pi_1; \pi_M \circ \delta)$. 

The space $\mathrm{Lifts}(\tau \times f, \theta_M \times \pi_M; \delta)$ is always non-empty. Finding a point in it is the same as finding a bundle injection $T\Sigma_{g,b} \hookrightarrow f^*TM$ extending $\delta$. The bundle $f^*TM$ is always trivial (it is a Spin vector bundle on a 2-manifold), and choosing a trivialisation $\eta : f^*TM \to \epsilon^3$, we see that finding such a bundle injection is the same as finding a lift in a diagram
\begin{diagram}
\partial \Sigma_{g,b} & \rTo^{\eta \circ \delta} & \mathrm{Gr}_2^+(\bR^3) & \,\,\simeq S^2\\
\dInto & & \dTo\\
\Sigma_{g,b} & \rTo^\tau & BSO(2).
\end{diagram}
By obstruction theory there is a unique obstruction $w_2(\tau) \in H^2(\Sigma_{g,b}, \partial \Sigma_{g,b};\bZ/2)$ to the existence of such a lift, but this is of course zero, as all orientable surfaces are Spin. Hence there is a surjection
$$\pi_0(\Bun_\partial(T\Sigma_{g,b}, \ell;\delta)) \lra \pi_0(\map_\partial(\Sigma_{g,b}, M;\pi_M \circ \delta)) \cong H_2(M;\bZ)$$
and the preimage of $[f]$ is the quotient of $\pi_0(\mathrm{Lifts}(\tau \times f, \theta_M \times \pi_M;\delta))$ by the action of $\pi_1(\map_\partial(\Sigma_{g,b}, M; \pi_M \circ \delta), \{f\})$.

In the fibration
$$\mathrm{Fr}_2(\bR^3) \lra \mathrm{Gr}_2^+(\gamma_3^{\Spin}) \overset{\pi_1 \times \pi_2}\lra BSO(2) \times B\Spin(3),$$
the space $\mathrm{Gr}_2^+(\gamma_3^{\Spin})$ has a sequence of bundles 
$$0 \lra \pi_1^*\gamma_2^{+} \lra \pi_2^*\gamma_3^{\Spin} \lra \epsilon^1 \lra 0$$
which splits to give an isomorphism $\pi_1^*\gamma_2^+ \oplus \epsilon^1 \cong \pi_2^*\gamma_3^{\Spin}$. Thus $\pi_1^*\gamma_2^+$ in fact has a canonical Spin structure, and $\pi_1$ factors canonically through $B\Spin(2)$. Given a map $f : \Sigma_{g,b} \to M$ with boundary condition $\pi_M \circ 
\delta$, there is a diagram
\begin{diagram}
 &  & \mathrm{Fr}_2(\bR^3) & \rEq & \mathrm{Fr}_2(\bR^3) & \rTo & \bR\bP^\infty\\
 &  & \dTo & & \dTo & & \dTo\\
\partial \Sigma_{g,b} & \rTo^\delta & \mathrm{Gr}_2^+(TM) & \rTo & \mathrm{Gr}_2^+(\gamma_3^{\Spin}) & \rTo & B\Spin(2)\\
\dInto & \ruTo^l[dotted]  & \dTo_{\theta_M \times \pi_M} &  & \dTo & & \dTo \\
\Sigma_{g,b} & \rTo^{\tau \times f} & BSO(2) \times M & \rTo^{\mathrm{Id} \times \mathfrak{s}} & BSO(2) \times B\Spin(3) & \rTo^{\pi_1} & BSO(2)
\end{diagram}
and hence a lift $l$ produces a Spin structure on $\Sigma_{g,b}$, and we have defined a map
$$\pi_0(\mathrm{Lifts}(\tau\times f, \theta_M \times \pi_M)) \lra \pi_0(\Bun_\partial(T\Sigma_{g,b}, \gamma_2^{\Spin};*)) =: \mathrm{Spin}(\Sigma_{g,b};*).$$
This is easily seen to be a bijection (as $\mathrm{Fr}_2(\bR^3) \simeq \bR\bP^3$, and the induced map $\mathrm{Fr}_2(\bR^3) \to \bR\bP^\infty$ between fibres of the rightmost two vertical maps in the diagram is the standard inclusion, hence 3-connected), and the action of a self-homotopy of $f$, that is, an element of $\pi_1(\map_\partial(\Sigma_{g,b}, M;\pi_M \circ b), \{f\})$, is trivial. Thus there is the exact sequence of sets
$$* \lra \mathrm{Spin}(\Sigma_{g,b};*) \lra \pi_0(\Bun_\partial(T\Sigma_{g,b}, \ell;\delta)) \lra H_2(M;\bZ) \lra *,$$
but a bundle map $T\Sigma_{g, b} \to \ell$ gives in particular a map $f$ and a lift of $\tau \times f$, so there is a function
$$\pi_0(\Bun_\partial(T\Sigma_{g,b}, \ell;\delta)) \lra \mathrm{Spin}(\Sigma_{g,b};*)$$
which splits it.
\end{proof}

\begin{prop}\label{prop:dim4}
Let $M$ be a simply-connected 4-manifold. Then there is a surjection
$$\pi_0(\Bun_\partial(T\Sigma_{g,b}, \ell;\delta)) \lra H_2(M;\bZ)$$
with preimage of $f$ in natural bijection with the set
$$P(f) := \{a \in \bZ\,\, |\,\, a \equiv \langle w_2(M), f \rangle \,\,\mathrm{mod}\,\, 2\}.$$
Thus $\pi_0(\Bun_\partial(T\Sigma_{g,b}, \ell;\delta))$ may be identified with a subset of $\bZ \times H_2(M;\bZ)$. The action of the mapping class group is trivial.
\end{prop}
\begin{proof}
Recall that $\mathrm{Fr}_2(\bR^4) $ is the homogenous space $ SO(2) \backslash SO(4)$, and that $\mathrm{Gr}_2^+(\gamma_4^{+})$ is the two-sided homotopy quotient
$$\mathrm{Gr}_2^+(\gamma_4^{+}) := SO(2)\backslash\!\!\backslash (SO(2) \backslash SO(4)) \hcoker SO(4) \simeq BSO(2) \times BSO(2).$$
A map giving this homotopy equivalence is given as follows: the space $\mathrm{Gr}_2^+(\gamma_4^{+})$ carries a sequence of bundles
$$0 \lra \pi_1^*\gamma_2^{+} \lra \pi_2^*\gamma_4^{+} \lra V \lra 0$$
so there is a map $c_V : \mathrm{Gr}_2^+(\gamma_4^{+}) \to BSO(2)$ classifying the oriented rank 2 bundle $V$. The map $\pi_1^*\gamma_2^{+} \times c_V : \mathrm{Gr}_2^+(\gamma_4^{+}) \to BSO(2) \times BSO(2)$ provides the homotopy equivalence.

Given a map $f : \Sigma_{g,b} \to M$ with boundary condition $\pi_M \circ 
\delta$, there is a diagram
\begin{diagram}
 &  & \mathrm{Fr}_2(\bR^4) & \rEq & \mathrm{Fr}_2(\bR^4) &  &\\
 &  & \dTo & & \dTo & & \\
\partial \Sigma_{g,b} & \rTo^\delta & \mathrm{Gr}_2^+(TM) & \rTo & \mathrm{Gr}_2^+(\gamma_4^{+}) & \rTo^{c_V} & BSO(2)\\
\dInto & \ruTo[dotted]^l  & \dTo_{\theta_M \times \pi_M} &  & \dTo_{\pi_1 \times \pi_2} & &\\
\Sigma_{g,b} & \rTo^{\tau \times f} & BSO(2) \times M & \rTo^{\mathrm{Id} \times \tau_M } & BSO(2) \times BSO(4) &  &
\end{diagram}
and hence composing a lift $l$ with $c_V$ produces an element $c(l) \in H^2(\Sigma_{g,b}, \partial \Sigma_{g,b};\bZ)$, that is, an integer. It reduces modulo 2 to $f^*w_2(M) \in H^2(\Sigma_{g, b}, \partial\Sigma_{g,b};\bZ/2) = \bZ/2$, which is a homotopy invariant of the map $f$. Furthermore, the set of homotopy classes of lifts is a $\pi_2(\mathrm{Fr}_2(\bR^4))$-torsor, and the image of
$$\bZ \cong \pi_2(\mathrm{Fr}_2(\bR^4)) \lra \pi_2(\mathrm{Gr}_2^+(\gamma_4^+)) \overset{c_V}\lra \pi_2(BSO(2)) = \bZ,$$
is the even integers. Thus we have defined an injective map
\begin{align}\label{eq:Dim4}
\pi_0(\mathrm{Lifts}(\tau\times f, \theta_M \times \pi_M;\delta)) &\lra P(f) :=
\begin{cases}
1+2\bZ &\text{if $f^*w_2(M) \neq 0$}\\
2\bZ &\text{if $f^*w_2(M) = 0$}
\end{cases} \subset \bZ\\
\nonumber l &\longmapsto c(l).
\end{align}
Given an oriented rank 2 vector bundle $V \to \Sigma_{g,b}$ trivialised over the boundary, standard obstruction theory shows that the only obstruction to finding an isomorphism $T\Sigma_{g,b} \oplus V \cong f^{*}TM$ (extending the standard isomorphism of trivial bundles on the boundary) is the class
$$f^*w_2(TM) - w_2(V)-w_2(T\Sigma_{g,b}) \in H^2(\Sigma_{g,b}, \partial \Sigma_{g,b};\bZ/2).$$
This shows that the map (\ref{eq:Dim4}) is also surjective.
\end{proof}

\begin{cor}\label{cor:pi0stab}
Let $M$ be simply-connected. If $\mathrm{dim}(M) = 3$ then choosing a Spin structure on $M$ gives a bijection
$$\pi_0(\mathcal{M}^{\theta_M}(\Sigma_{g, b};\delta)) \cong \pi_0(\mathcal{M}^{\Spin}(\Sigma_{g, b};*)) \times H_2(M;\bZ)$$
for all $g$ and $b$. Gluing surfaces along boundary components corresponds to adding homology classes and gluing Spin structures. 

If $\mathrm{dim}(M) = 4$ then there is a surjection
$$\pi_0(\mathcal{M}^{\theta_M}(\Sigma_{g, b};\delta)) \lra H_2(M;\bZ)$$
with preimage of $f$ given by the set $P(f)$, for all $g$ and $b$. Gluing surfaces $(a \in P(f), f)$ and $(b \in P(g), g)$ along boundary components has result $(a+b \in P(f+g), f+g)$.

If $\mathrm{dim}(M) \geq 5$ then there is a bijection
$$\pi_0(\mathcal{M}^{\theta_M}(\Sigma_{g, b};\delta)) \cong H_2(M;\bZ)$$
for all $g$ and $b$. Gluing surfaces along boundary components corresponds to adding homology classes.
\end{cor}
\begin{proof}
The previous three propositions provide the descriptions of the sets. Gluing together immersed surfaces certain adds the homology classes they represent. In dimension 3 the Spin structure induced on the union of two immersed surfaces is the union of the Spin structures on each, and in dimension 4 the Euler number of the normal bundle of an immersion of a union of two surfaces is the sum of the Euler numbers of the individual immersions.
\end{proof}

Using these calculations and the methods of \cite{R-WResolution}, we establish the following homology stability theorem for the spaces $\mathcal{M}^{\theta_M}(\Sigma_{g,b};\delta)$. We must first define certain basic stabilisation maps. Write
$$\alpha(g) : \mathcal{M}^{\theta_M}(\Sigma_{g, b};\delta) \lra \mathcal{M}^{\theta_M}(\Sigma_{g+1, b-1};\delta')$$
for any gluing map which adds on a pair of pants along the legs (this is only defined for $b \geq 2$). Write
$$\beta(g) : \mathcal{M}^{\theta_M}(\Sigma_{g, b};\delta) \lra \mathcal{M}^{\theta_M}(\Sigma_{g, b+1};\delta')$$
for any gluing map which adds on a pair of pants along the waist. Write
$$\gamma(g) : \mathcal{M}^{\theta_M}(\Sigma_{g, b};\delta) \lra \mathcal{M}^{\theta_M}(\Sigma_{g, b-1};\delta')$$
for any gluing map which adds on a disc along some boundary component.

\begin{thm}\label{thm:stability}
Any stabilisation map $\alpha(g)$ is a homology epimorphism in degrees
$$* \leq \begin{cases}
\tfrac{2g-1}{5} & \text{if $\dim(M) = 3$}\\
\tfrac{2g}{3} & \text{if $\dim(M) > 3$}
\end{cases}$$
and a homology isomorphism in one degree lower.

Any stabilisation map $\beta(g)$ is a homology epimorphism in degrees
$$* \leq \begin{cases}
\tfrac{2g-2}{5} & \text{if $\dim(M) = 3$}\\
\tfrac{2g-1}{3} & \text{if $\dim(M) > 3$}
\end{cases}$$
and a homology isomorphism in one degree lower.

Any stabilisation map $\gamma(g)$ is a homology isomorphism in the same range that $\beta(g)$ is an epimorphism.
\end{thm}


\begin{proof}
By Corollary \ref{cor:pi0stab} if $\dim(M) \geq 4$ then all stabilisation maps induce bijections on sets of path components, and if $\dim(M)=3$ stabilisation maps induce surjections or bijections on sets of path components precisely when Spin structures do. Thus by \cite[\S 7]{R-WResolution} these cases have the same stability ranges as the tangential structures $BSO(2) \times M$ and $B\Spin(2) \times M$ respectively. The range in the first case has been calculated in \cite[\S 7.5]{R-WResolution}, and the range in the second case follows from \cite[\S 2.4-2.6]{RWFramedPinMCG}.
\end{proof}

To identify the stable homology, we apply the theorem of Galatius--Madsen--Tillmann--Weiss \cite[\S 7]{GMTW}. Write $\MT{SO}{2}$ for the Thom spectrum of the virtual bundle $-\gamma_2 \to BSO(2)$. The map of Thom spectra induced by $\theta_M : \mathrm{Gr}_2^+(TM) \to BSO(2)$ gives a map
$$\Omega^\infty\MTtheta_M \lra \Omega^\infty\MT{SO}{2}$$
and it has been calculated in \cite{MT} that there is a natural isomorphism
$$E : \pi_0(\Omega^\infty\MT{SO}{2}) \cong \bZ$$
under which the point represented by an oriented surface of genus $g$ maps to $1-g$. The evident composition defines a map
$$E : \pi_0(\Omega^\infty\MTtheta_M) \lra \bZ$$
and we let $\Omega^\infty_n \MTtheta_M$ be those path components which map to $n \in \bZ$ under $E$. In \cite[\S 5]{GMTW} the authors show that Pontrjagin--Thom theory provides a map
$$\alpha_M : \mathcal{M}^{\theta_M}(\Sigma_g) \lra \Omega^\infty_{1-g} \MTtheta_M,$$
and by Theorem \ref{thm:stability} and \cite[\S 11]{R-WResolution}, this map
is a homology isomorphism in degrees
$$* \leq \begin{cases}
\tfrac{2g-6}{5} & \text{if $\dim(M) = 3$}\\
\tfrac{2g-3}{3} & \text{if $\dim(M) > 3$},
\end{cases}$$
which establishes Theorem \ref{thm:Main}. 

\section{Calculations in Euclidean space}\label{sec:CalcEuc}

In this section we prove the results of Sections \ref{sec:3}--\ref{sec:2n}. 

\subsection{Immersions in $\bR^3$}

That $\mathcal{I}_g(\bR^3)$ has two components, distinguished by the Arf invariant of the associated Spin structures, follows from Proposition \ref{prop:dim3}. In this case $\mathrm{Gr}_2^+(\bR^3) \simeq S^2$ and the tautological bundle corresponds to $TS^2$. This is stably trivial and so $\MTtheta_{\bR^3} \simeq \mathbf{S}^{-2} \vee \mathbf{S}^0$. Hence the associated infinite loop space is $Q_0(S^0) \times \Omega^2_0 Q(S^0)$, which has trivial rational cohomology in positive degrees. The result now follows from Corollary \ref{cor:Main}.

The torsion calculation may be seen as follows. The fundamental group of this infinite loop space is $\bZ/2 \oplus \bZ/24$, by the well known homotopy groups of spheres in low degrees. This is its first integral homology too, by Hurewicz' theorem, and it follows from Corollary \ref{cor:Main} that this is the first homology of $\mathcal{I}_g^h(\bR^3)$ as long as $g \geq 6$. As the fibres of $\mathcal{I}_g^h(\bR^3) \to \mathcal{I}_g(\bR^3)$ are all connected, it follows that the map is surjective on first homology, and the statement about localisation follows from Theorem \ref{thm:Orbifold}.

\subsection{Immersions in $\bR^4$}

That $\mathcal{I}_g(\bR^4)$ has components indexed by $\bZ$, distinguished by the Euler number of the normal bundle of the immersion, follows from Proposition \ref{prop:dim4}. $\mathrm{Gr}_2^+(\bR^4)$ is a simply-connected 4-manifold (in fact, it is $S^2 \times S^2$ but we do not require this). Thus the cohomology of the spectrum $\MTtheta_{\bR^4}$ has a unique class in positive degree, $[\mathrm{Gr}_2^+(\bR^4)]^* \cdot u_{-2} \in H^2(\MTtheta_{\bR^4};\bQ)$, and so the rational cohomology of $\Omega^{\infty}_\bullet \MTtheta_{\bR^4}$ is
$$H^*(\Omega^{\infty}_\bullet \MTtheta_{\bR^4};\bQ) \cong \bQ[a_2]$$
a polynomial algebra on a single generator in degree 2. One can easily check that the natural map $\mathrm{Gr}_2^+(\bR^4) \to \mathrm{Gr}_2^+(\bR^\infty)$ pulls back the square of the Euler class to a non-trivial top-dimensional class, and so $a_2$ can be taken to be $\kappa_1$. The result now follows from Corollary \ref{cor:Main}.

\subsection{Immersions in $\bR^{2n+1}$, $2n+1 \geq 5$}

That $\mathcal{I}_g(\bR^{2n+1})$ is connected follows from Proposition \ref{prop:5mfld}. We have the calculation
$$H^*(\mathrm{Gr}_2^+(\bR^{2n+1});\bQ) = \bQ[e]/(e^{2n})$$
where $e \in H^2(\mathrm{Gr}_2^+(\bR^{2n+1});\bQ)$ is the Euler class. Thus
$$H^*(\Omega^{\infty}_\bullet \MTtheta_{\bR^{2n+1}};\bQ) \cong \bQ[\kappa_1, \kappa_2, ..., \kappa_{2n-2}]$$
and so the result now follows from Corollary \ref{cor:Main}.

\subsection{Immersions in $\bR^{2n}$, $2n \geq 6$}

That $\mathcal{I}_g(\bR^{2n+1})$ is connected follows from Proposition \ref{prop:5mfld}. By \cite[Theorem 2]{Lai} we have the calculation
$$H^*(\mathrm{Gr}_2^+(\bR^{2n});\bQ) = \bQ[e, \delta]/(\delta^2, e^n-2\delta e)$$
where $e$ is the Euler class and $\delta$ is Poincar{\'e} dual to the fundamental class of the submanifold $\bC\bP^{n-1} \hookrightarrow \mathrm{Gr}_2^+(\bR^{2n})$. In terms of these classes, the Euler class of the orthogonal complement to the tautological bundle is
$$\bar{e} = 2\delta - e^{n-1}.$$
We define a spectrum cohomology class $\bar{e} \cdot u_{-2} \in H^{2n-4}(\MTtheta_{\bR^{2n}};\bQ)$ and let $\Delta \in H^{2n-4}(\Omega^\infty_\bullet\MTtheta_{\bR^{2n}};\bQ)$ denote its cohomology suspension. Then 
$$H^*(\Omega^\infty_\bullet\MTtheta_{\bR^{2n}};\bQ) \cong \bQ[\kappa_1, \kappa_2, ..., \kappa_{2n-3}, \Delta]$$
(and $\kappa_i=0$ for $i > 2n-3)$, so the result now follows from Corollary \ref{cor:Main}.

\subsection{The spectral sequence for immersions in $\bR^{2n+1}$, $2n+1 \geq 5$}\label{sec:sseq}

In order to emphasise the nontriviality of the fibration
\begin{equation}\label{eq:ImmFib}
\mathrm{Imm}(\Sigma_g, M) \lra \mathcal{I}_g^h(M) \lra B\Diff^+(\Sigma_g),
\end{equation}
we will describe a computation of its Leray--Serre spectral sequence when $M = \bR^{2n+1}$ with $2n+1 \geq 5$, for rational cohomology in the stable range. We must first understand the cohomology of the space $\mathrm{Imm}(\Sigma_g, \bR^{2n+1})$, and the coefficient system it describes over $B\Diff^+(\Sigma_g)$.

Recall that the \emph{mapping class group} of a closed genus $g$ surface is defined to be
$$\Gamma_g := \pi_0(\Diff^+(\Sigma_g)),$$
and for $g \geq 2$ the map $B\Diff^+(\Sigma_g) \to B\Gamma_g$ is a homotopy equivalence \cite{EE}. The action of a diffeomorphism of $\Sigma_g$ on the homology of the surface makes $H_1(\Sigma_g ;\bZ)$ into a $\Gamma_g$-module, which we call $H$. We write $H_\bQ = H \otimes \bQ$. The naturality of Poincar{\'e} duality and the Universal Coefficient Theorem shows that $H \cong H^*$ as $\Gamma_g$-modules.

\begin{prop}
There is an isomorphism of $\Gamma_g$-modules
$$H^*(\mathrm{Imm}(\Sigma_g, \bR^{2n+1});\bQ) \cong \Lambda[x_{4n-3}, x_{4n-1}] \otimes \mathrm{Sym}^*(H_\bQ[4n-2])$$
where $H_\bQ[4n-2]$ denotes $H_\bQ$ considered as a vector space of grading $(4n-2)$, and $\mathrm{Sym}^*$ denotes the symmetric algebra on this graded vector space.
\end{prop}

\begin{proof}
We first decompose $\mathrm{Imm}(\Sigma_g, \bR^{2n+1})$ using the fibration
\begin{equation}\label{eq:Dec1}
\mathrm{Imm}_\partial(D^2, \bR^{2n+1}) \lra \mathrm{Imm}(\Sigma_g, \bR^{2n+1}) \overset{\text{restrict}}\lra \mathrm{Imm}(\Sigma_g \setminus D^2, \bR^{2n+1}),
\end{equation}
which restricts an immersion to the complement of a disc. We then consider the fibration
\begin{equation}\label{eq:Dec2}
\mathrm{Imm}_*(\Sigma_g \setminus D^2, \bR^{2n+1}) \lra \mathrm{Imm}(\Sigma_g \setminus D^2, \bR^{2n+1}) \overset{\text{restrict}}\lra \Fr_2(\bR^{2n+1}),
\end{equation}
where the base is the space of linearly-independent 2-frames in $\bR^{2n+1}$, which takes the derivative of an immersion at a point $x_0 \in \partial(\Sigma_g \setminus D^2)$. The notation $\mathrm{Imm}_*(\Sigma_g \setminus D^2, \bR^{2n+1})$ means the space of immersions which agree with a fixed germ near $x_0$. Both fibrations admit an action of $\Diff^+(\Sigma_{g, 1}, \partial)$. In the first case the action is trivial on each fibre, and in the second case it is trivial on the base. If we write $\Gamma_{g,1} := \pi_0(\Diff^+(\Sigma_{g,1}, \partial))$, this makes the Serre spectral sequence a spectral sequence of $\Gamma_{g,1}$-modules in both cases.

The fibration
$$S^{2n-1} \lra \Fr_2(\bR^{2n+1}) \lra S^{2n}$$
is equivalent to the sphere bundle of $T S^{2n}$. As such it has non-trivial Euler class, and so the Serre spectral sequence has a non-trivial differential. Thus we have $H^*(\Fr_2(\bR^{2n+1});\bQ) = \Lambda[x_{4n-1}]$.

We now apply Hirsch--Smale theory. There is a map
\begin{align*}
\mathrm{Imm}_\partial(D^2, \bR^{2n+1}) &\lra \map_{\partial}(D^2, \Fr_2(\bR^{2n+1}))\\
i & \longmapsto \epsilon^2 = TD^2 \overset{Di}\to T\bR^{2n+1}
\end{align*}
to the space of maps which have some fixed behaviour on the boundary, and by Hirsch--Smale theory this is an equivalence. As $\Fr_2(\bR^{2n+1})$ is simply-connected, we can suppose that the boundary condition is the constant map to a basepoint, and so there is a homotopy equivalence $\mathrm{Imm}_\partial(D^2, \bR^{2n+1}) \simeq \Omega^2 \Fr_2(\bR^{2n+1})$. 
Together with the calculation above, we deduce that
$$H^*(\mathrm{Imm}_\partial(D^2, \bR^{2n+1});\bQ) = \Lambda[x_{4n-3}].$$

We now study the cohomology of $\mathrm{Imm}_*(\Sigma_{g,1}, \bR^{2n+1})$, which is a little more complicated. Choosing a trivialisation $\varphi : T\Sigma_{g,1} \cong \epsilon^2$ defines a map
\begin{align*}
T_\varphi : \mathrm{Imm}_*(\Sigma_{g,1}, \bR^{2n+1}) & \lra \map_*(\Sigma_{g,1}, \Fr_2(\bR^{2n+1}))\\
i &\longmapsto \epsilon^2 \cong_{\varphi^{-1}} T\Sigma_{g,2} \overset{Di}\to T\bR^{2n+1}
\end{align*}
and by Hirsch--Smale theory this is a homotopy equivalence. Both sides have an action of the group $\Diff^+(\Sigma_{g,1}, \partial)$, but the map $T_\varphi$ is not equivariant for this action. To describe the failure of equivariance, consider the map
\begin{align*}
\psi : \Diff^+(\Sigma_{g,1}, \partial) & \lra \map_\partial(\Sigma_{g,1}, GL^+_2(\bR))\\
f &\longmapsto \epsilon^2 \cong_{\varphi^{-1}} T\Sigma_{g,1} \overset{f}\to T\Sigma_{g,1} \cong_\varphi \epsilon^2.
\end{align*}
Let us write $f \cdot -$ for the action of $f$ on $\mathrm{Imm}_*(\Sigma_{g,1}, \bR^{2n+1})$, $f * -$ for the action on $\map_*(\Sigma_{g,1}, \Fr_2(\bR^{2n+1}))$, and $\psi(f) \circ -$ for the action of $\psi(f)$ on $\map_*(\Sigma_{g,1}, \Fr_2(\bR^{2n+1}))$ induced by the action of $GL^+_2(\bR)$ on $\Fr_2(\bR^{2n+1})$. Then we have the relationship $T_\varphi(f \cdot i) = \psi(f) \circ (f * T_\varphi(i))$ between these actions.

We have the homotopy equivalence
$$\map_*(\Sigma_{g,1}, \Fr_2(\bR^{2n+1})) \simeq [\Omega \Fr_2(\bR^{2n+1})]^{2g}$$
and so $H^*(\map_*(\Sigma_{g,1}, \Fr_2(\bR^{2n+1}));\bQ) = \bQ[a^1_{4n-2}, b^1_{4n-2}, ..., a^g_{4n-2}, b^g_{4n-2}]$ and the action $f*-$ of a diffeomorphism is the usual symplectic action on the variables $a^i, b^i$. Thus $H^*(\map_*(\Sigma_{g,1}, \Fr_2(\bR^{2n+1}));\bQ) = \mathrm{Sym}^*(H_\bQ[4n-2])$ as a $\Gamma_{g,1}$-module.

This computes $H^*(\mathrm{Imm}_*(\Sigma_{g,1}, \bR^{2n+1});\bQ)$ as a ring, but we must compute the action $f \cdot -$ as well. By the formula above, this corresponds to computing the action $\psi(f) \circ -$ on $H^*(\map_*(\Sigma_{g,1}, \Fr_2(\bR^{2n+1}));\bQ)$ induced by $\mu : GL^+_2(\bR) \times \Fr_2(\bR^{2n+1}) \to \Fr_2(\bR^{2n+1})$. However for dimension reasons the action $\mu$ is trivial on homology and it is easy to deduce from this that $\psi(f) \circ -$ acts trivially too. Thus despite $T_\varphi$ not being an equivariant map, the map $(T_\varphi)^*$ is a map (in fact, an isomorphism) of $\Gamma_{g,1}$-modules.

Consider the Serre spectral sequence for the fibration (\ref{eq:Dec2}), which has the form
$$\mathrm{Sym}^*(H_\bQ[4n-2]) \otimes \Lambda[x_{4n-1}] \Longrightarrow H^*(\mathrm{Imm}(\Sigma_g \setminus D^2, \bR^{2n+1});\bQ).$$
The only possible differential is $d_{4n-1}$, and it is determined by
$$d_{4n-1} : E^{0, 4n-2} = H_\bQ \lra E^{4n-1, 0} = \bQ$$
but is also a map of $\Gamma_{g,1}$-modules, so must be zero (it corresponds to an invariant vector in $H^*_\bQ \cong H_\bQ$). Thus the spectral sequence collapses, and one may check that there are no extensions.

Next we consider the Serre spectral sequence for the fibration (\ref{eq:Dec1}), which has the form
$$\Lambda[x_{4n-3}] \otimes \mathrm{Sym}^*(H_\bQ[4n-2]) \otimes \Lambda[x_{4n-1}] \Longrightarrow H^*(\mathrm{Imm}(\Sigma_g, \bR^{2n+1});\bQ).$$
The only possible differential is $d_{4n-2}$, and it is determined by
$$d_{4n-2} : E^{0, 4n-1} = \bQ \lra E^{4n-2, 0} = H_\bQ$$
but this must again be trivial as it corresponds to a $\Gamma_{g,1}$-invariant vector in $H_\bQ$. Thus this spectral sequence also collapses, and again one may check that there are no extensions. This determines $H^*(\mathrm{Imm}(\Sigma_g, \bR^{2n+1});\bQ)$ as a $\Gamma_{g,1}$-module, and the final step is to observe that the natural homomorphism $\Gamma_{g,1} \to \Gamma_g$ is surjective, and so we have determined the $\Gamma_g$-module structure.
\end{proof}

Hence (for $g \geq 2$) the Leray--Serre spectral sequence for the fibration (\ref{eq:ImmFib}) has the form
\begin{equation}\label{eq:sseq}
H^*(\Gamma_g ; \Lambda[x_{4n-3}, x_{4n-1}] \otimes \mathrm{Sym}^*(H_\bQ[4n-2])) \Longrightarrow H^*(\mathcal{I}^h_g(\bR^{2n+1});\bQ).
\end{equation}
To determine the $E_2$-term of this spectral sequence we must compute the cohomology of $\Gamma_g$ with coefficients in $\mathrm{Sym}^q(H_\bQ)$, and it is very useful to have the multiplicative structure, induced by $\mathrm{Sym}^p(H_\bQ) \otimes \mathrm{Sym}^q(H_\bQ) \to \mathrm{Sym}^{p+q}(H_\bQ)$, available to us too.

\begin{prop}
There is an isomorphism of bigraded algebras
$$\bigoplus_{p,q} H^p(\Gamma_g;\mathrm{Sym}^q(H_\bQ)) \cong \bQ[\kappa_1, \kappa_2, ...] \otimes \Lambda[\ell_3, \ell_5, \ell_7, ...],$$
in total degrees $p+q \leq \tfrac{2g-3}{3}$, where $\kappa_i$ has bi-degree $(p,q)=(2i,0)$ and $\ell_{2i+1}$ has bi-degree $(p,q)=(2i+1, 1)$.
\end{prop}
\begin{proof}[Proof sketch]
By a theorem of Looijenga \cite[Example 1]{LooijengaAJ} there is an isomorphism
$$H^*(\Gamma_g;\mathrm{Sym}^s (H_\bQ)) \cong \bQ[\kappa_1, \kappa_2, ...] \otimes \Sigma^{s(s+2)} \bQ[x_2, ..., x_{2s}]$$
of $\bQ[\kappa_1, \kappa_2, ...]$-modules in the stable range, i.e.\ the stable cohomology is a free $\bQ[\kappa_1, \kappa_2, ...]$-module on one generator for each monomial of $\bQ[x_2, ..., x_{2s}]$, with suitably shifted degrees. Some dilligent work with generating functions\footnote{As a hint to the reader, after a change of variables the generating functions for the ranks as a $\bQ[\kappa_1, \kappa_2, ...]$-module of Looijenga's and our description become the left- and right-hand sides respectively of the following identity
$$\sum_{k=0}^\infty q^{k(k-1)/2} \frac{1}{1-q} \cdots \frac{1}{1-q^k} \cdot q^kx^k = \prod_{k=1}^\infty(1+q^kx),$$ which holds as the right-hand side is nothing but the $q$-Pochhammer symbol $(-x;q)_\infty$ and the left-hand side is its well-known series expansion.}---which we omit---shows that the statement of the proposition is true as $\bQ[\kappa_1, \kappa_2, ...]$-modules: there are the correct number of free generators in each degree.

What remains is to show that the multiplicative structure is as claimed, which we will just sketch briefly: the argument is analogous to that of Ebert and the author \cite{ERW10} where we compute the multiplicative structure of $H^*(\Gamma_g; \Lambda^*H_\bQ)$ in the stable range. We consider the moduli space $\mathcal{S}_{g, b} := \map_\partial(\Sigma_{g, b}, K(\bZ, 3);*) \hcoker \Diff^+_\partial(\Sigma_{g,b})$ which classifies surface bundles with a third integral cohomology class on the total space. By a theorem of Cohen and Madsen \cite{CM, CM2}, and the extension to closed surfaces by the author \cite{R-WResolution}, these spaces have homological stability with homology independent of $g$ in degrees $* \leq \frac{2g-3}{3}$. The stable homology is that of the infinite loop space $\Omega^\infty_\bullet(\MT{SO}{2} \wedge K(\bZ, 3)_+)$. Standard methods give the calculation
$$H^*(\Omega^\infty_\bullet(\MT{SO}{2} \wedge K(\bZ, 3)_+);\bQ) \cong \bQ[\kappa_1, \kappa_2, ...] \otimes \Lambda[l_1, l_3, l_5, ...],$$
for some classes $l_i$ of degree $i$. Next we observe that there is an unnatural decomposition
$$\map(\Sigma_g, K(\bZ, 3)) \simeq K(\bZ, 3) \times K(H, 2) \times K(\bZ, 1)$$
as the space of maps into an Eilenberg--MacLane space is again a product of Eilenberg--MacLane spaces, and it is easy to calculate its homotopy groups. Thus the Serre spectral sequence for the fibration
$$\map(\Sigma_g, K(\bZ, 3)) \lra \mathcal{S}_g \lra B\Diff^+(\Sigma_g)$$
has the form
$$H^*(\Gamma_g; \mathrm{Sym}^*(H_\bQ[2])) \otimes \Lambda[x_1, x_3] \Longrightarrow \bQ[\kappa_1, \kappa_2, ...] \otimes \Lambda[l_1, l_3, l_5, ...]$$
in total degrees $* \leq \frac{2g-3}{3}$. By counting dimensions we see that this spectral sequence must collapse in the stable range, and so the multiplicative structure is as claimed (with $\ell_{2i+3}$ corresponding to $l_{2i+5}$).
\end{proof}

Thus the spectral sequence (\ref{eq:sseq}) has the form
$$\bQ[\kappa_1, \kappa_2, ...] \otimes \Lambda[x_{4n-3}, x_{4n-1}, \ell_{4n+1}, \ell_{4n+3}, ...] \Longrightarrow H^*(\mathcal{I}^h_g(\bR^{2n+1});\bQ),$$
in the stable range, where $\ell_{4n+1 + 2k}$ has bi-degree $(p, q)=(3+2k, 4n-2)$. A chart of the spectral sequence is then as shown in Figure \ref{fig:1}, where each dot represents a free $\bQ[\kappa_1, \kappa_2, ...]$-module generator. 
\begin{figure}[h]
\centering
\includegraphics[bb = 0 0 357 212]{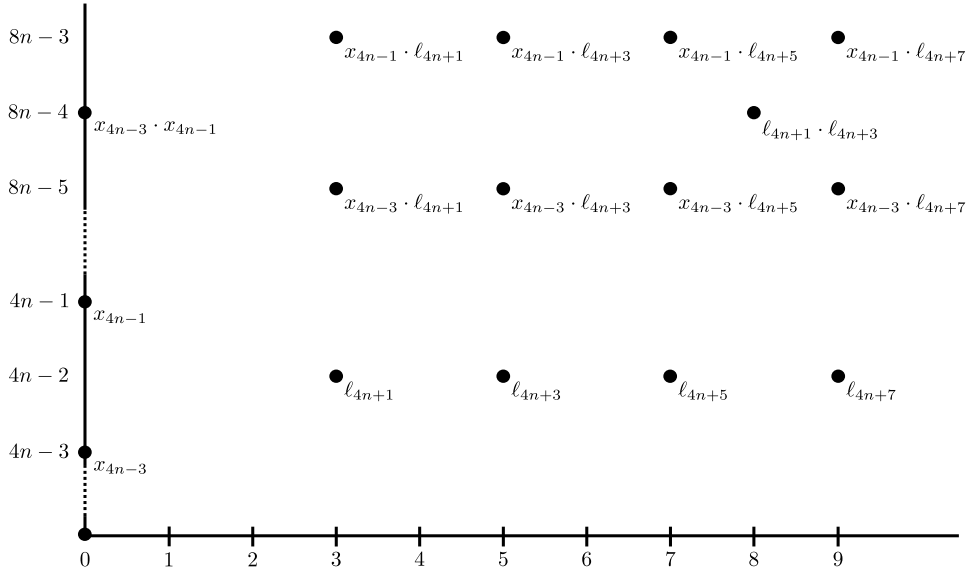}
\caption{Chart of the spectral sequence (\ref{eq:sseq}).}\label{fig:1}
\end{figure}
By the result of Section \ref{sec:2np1}, we know that
$$H^*(\mathcal{I}^h_g(\bR^{2n+1});\bQ) \cong \bQ[\kappa_1, ..., \kappa_{2n-2}]$$
in degrees $* \leq \tfrac{2g-3}{3}$. If we suppose that $g \gg n$, then the classes $\kappa_{2n-1}, \kappa_{2n}, ...$ cannot survive the spectral sequence. We see from the chart that the only possible pattern of differentials  in the stable range is
\begin{align*}
d_{4n-2}(x_{4n-3}) &= \kappa_{2n-1}\\
d_{4n-1}(\ell_{4n+1 + 2k}) &= \kappa_{2n+1+k} \quad\text{for}\,\, k \geq 0\\
d_{4n}(x_{4n-1}) &= \kappa_{2n}
\end{align*}
up to units.

\bibliographystyle{amsalpha}
\bibliography{MainBib}

\end{document}